\documentclass[12pt]{article}
\usepackage{amsfonts,amsmath,amsthm, amssymb}
\usepackage{latexsym, euscript, epic, eepic}
\newtheorem{cro}{Corollary}[section]
\newtheorem{defn}{Definition}[section]
\newtheorem{prop}{Proposition}[section]
\newtheorem{thm}{Theorem}[section]

\newtheorem{rem}{\bf Remark}[section]

\begin{document}
 \title{ Quantitative recurrence properties for systems with non-uniform structure
\footnotetext {
*Corresponding author\\
2010 Mathematics Subject Classification: 37D25, 37D35
  }}
\author{Cao Zhao$^{1},$ Ercai Chen$^{1,2_{*}}.$ \\
   \small   1 School of Mathematical Science, Nanjing Normal University,\\
    \small   Nanjing 210023, Jiangsu, P.R. China\\
     \small 2 Center of Nonlinear Science, Nanjing University,\\
         \small   Nanjing 210093, Jiangsu, P.R. China\\
          \small    e-mail: ecchen@njnu.edu.cn \\
           \small    e-mail: izhaocao@126.com \\}

\date{}
\maketitle

\begin{center}
 \begin{minipage}{120mm}
{\small {\bf Abstract.}   Let $X$ be a subshift satisfy non-uniform structure, and $\sigma:X\to X$ be a shift map. Further, define $$R(\psi) := \{x \in X : d(\sigma^{n}x, x)<\psi(n) ~\text{for infinitely many}~ n\}£¬ $$
and $$R(f)  := \{x \in X : d(\sigma^{n}x, x)<e^{ S_{n}f(x)}  ~\text{for infinitely many}~ n\}, $$
where $\psi : \mathbb N \to \mathbb R^{+}  $ is a  nonincreasing and positive function, and $f: X\to \mathbb R^{+}$ is a continuous positive function. In this paper, we give quantitative estimate of the above sets, that is, ${\dim}_{H}R(\psi)$ can be expressed by $\psi$ and ${\dim}_{H}R(f)$ is the solution of the $Bowen $ equation of topological pressure.  These results can be applied to a large class of symbolic systems, including $\beta$-shifts, $S$-gap shifts and their factors. }
\end{minipage}
 \end{center}
\section{ Introduction.}
Let $(X,T, d)$ be a topological dynamical system, where $(X,d)$ is a compact metric space and $T:X\to X$ a continuous map. The set $M(X)$ denotes all Borel probability measures is a compact space for the weak$^{*}$ topology of measures, and $M(X,T)$ is the subset of $T$-invariant probability measures with the induced topology, and then for any $\mu\in M(X, T)$, define $(X,T, \mu, d)$ be a measure dynamical system. For the measure dynamical system $(X,T, d,\mu)$, Poincar\'{e} Recurrence Theorem states that given any  invariant measure, almost every point in any positive measure set $E$ returns $E$ an infinite number of times. These results are qualitative in nature, there are fruitful results about the descriptions of the recurrence in this way. We refer the reader to
\cite{Fur} and the references therein.
While, these results does not address either with which rate the orbit will return back to the initial point or in which manner the neighborhood of the initial point can shrink.
This, later, in \cite{Bos}, Boshernitzan presented the following result for general systems.
\begin{thm}{\rm \cite{Bos}}
Let $(X,T,\mu, d)$ be a measure dynamical system. Assume that, for some $\alpha>0$, the $\alpha$-dimensional Hausdorff measure $H^{\alpha}$ of the space $X$ is $\sigma$-finite. Then for $\mu$-almost all $x\in X,$
$$\liminf\limits_{n\to\infty}n^{\frac{1}{\alpha}}d(T^{n}x, x)<\infty.$$
If, moreover, $H^{\alpha}(X)=0, $ then for $\mu$-almost all $x\in X$,
$$\liminf\limits_{n\to\infty}n^{\frac{1}{\alpha}}d(T^{n}x, x)=0.$$
\end{thm}
Furthermore, Barreira and Saussol study the shrinking rate of the local pointwise dimension.
\begin{thm}\label{BarSau}
Let $T:X\to X$ be a Borel measure transformation on a measurable set $X\subset \mathbb R^{m}$ for some $m\in \mathbb N$, and $\mu$ be a $T$-invariant probability measure on $X$. Then $\mu$-almost surely,
$$\lim\limits_{n\to\infty}n^{\frac{1}{\alpha}}d(T^{n}x, x)<\infty$$
for any $\alpha>\underline{d}_{\mu}(x),$ where $\underline{d}_{\mu}(x)$ is the lower pointwise dimension of $x$ with respect to $\mu$, given as
$$\underline{d}_{\mu}(x)=\liminf\limits_{r\to 0}\frac{\log \mu(B(x,r))}{\log r}.$$
\end{thm}
Clearly, Boshernitzan showed almost all point have low recurrence rate. And Barreira and Saussol showed that the shrinking rate for recurrence may relate to some indicators of $x$. Another direction is that how large will the set of points be when the shrinking rate for recurrence is related to other funcitons?
In \cite{Hil, Hil1,Hil2}, Hill introduced a shrinking target problems from number theory and gave quantitative research of the recurrence.
Let $T:\mathbf J\to \mathbf J$ be an expanding rational map of the Riemann sphere acting on
its Julia set $\mathbf J$ and $f : \mathbf J \to  \mathbb R $ denote a H\"{o}lder continuous function satisfying $f(x)\geq
\log |T^{'}(x)|$ for all $x\in \mathbf J$. Then for any $z_{0}\in \mathbf J$, in \cite{Hil}, Hill and Velani studied the set
of "well approximable" points
$$D_{z_{0}}(f) := \{x\in \mathbf J  : d(y, x) < e^{-S_{n}f(y)}  ~\text{for infinitely many pairs }~ (y, n)\in \mathbf I\},$$
where $\mathbf I = \mathbf I(z_{0}) $ denotes the set of pairs $(y, n)(n\in \mathbb N)$ such that $T^{n}y=z_{0}$. In fact,
they gave the following result.
\begin{thm}
The set $D_{z_{0}}(f)$
  has Hausdorff dimension $s(f)$, where $s(f)$ is the unique
solution to the pressure equation
$$P(-s f )= 0.$$
\end{thm}
In \cite{TanWan}, Tan and Wang investigated metric properties as well as estimations on the Hausdorff dimension of the recurrence set for $\beta$-transformation dynamical systems. More precisely,
 the $\beta$-transformation $T_{\beta}:[0,1]\to [0,1]$ is defined by $T_{\beta}=\beta x- \lfloor  \beta x\rfloor $  for all $x\in [0,1]$. And the spotlight is on the size of the set
$$ \{x \in X : d(T_{\beta}^{n}x, x)<\psi(n) ~\text{for infinitely many}~ n\},  $$
where $\psi(n)$ is a positive function. In fact, this evokes a rich subsequent work on the so-called Diophantine approximation, we refer the reader to \cite{Hil, TanWan, Wan, Fan} for the related work about this set. That is worth mentioning, the research of recurrence for Diophantine approximation are concentrated in the $\beta$-transformation. In other words, the question is whether we can give an quantitative estimate of recurrence for a more general dynamical systems.

In this paper, we consider a class of symbolic systems which is studied in \cite{CliTho}. That is, $(X,\sigma)$ is a symbolic system with non-uniform structure for the symbolic systems $(X,\sigma)$.  The non-uniform structure mainly defined as there exist  $\mathcal{G}\subset \mathcal{L}(X)$ has $(W)$-specification and  $\mathcal{L}(X)$ is edit approachable by $\mathcal{G}$. The detail definitions will be given in the next section.

Our main result is the following. Set the symbolic system $(X,\sigma, d)$  with $\sigma : ¡¡X \to X$ is shift map, and $d$ is the metric of $X$.
Respectively, $M(X), M_{\sigma}(X)$ denote the probability measure and invariant measure with weak$^{*}$ topology.
We define
$$ R(\psi):=\{x\in X: d(\sigma^{n}x,x)<\psi(n)~\text{for infinitely many}~n\in \mathbb N\} .$$

\begin{thm}\label{main1}
Let $X$ be a shift space with $\mathcal{L}=\mathcal{L}(X)$. Suppose that $\mathcal{G}\subset \mathcal{L}$ has $(W)$-specification and  $\mathcal{L}$ is edit approachable by $\mathcal{G}$, then for positive function $\psi(n):\mathbb N\to \mathbb R$:

{\bf C1)}If $\liminf\limits_{n\to\infty}\psi(n)>0$, then

 $${\rm dim}_{H}R(\psi)= h   ,$$

{\bf  C2)}If $\psi$ is nonincreasing, then

$${\rm dim}_{H}R(\psi)=\dfrac{h}{1+b},~\text{with}~b=\liminf\limits_{n\to\infty}\dfrac{-\log \psi(n)}{n},$$
where ${\rm dim}_{H}$ denotes the Hausdorff dimension of a set and $h:=h_{top}(X).$
\end{thm}

\noindent Let $f$ be a positive continuous function defined on $X$, set
$$R(f)=\{x\in X:~d(\sigma^{n}x,x)\leq e^{-S_{n}f(x)} ~\text{for infinitely many } ~n\in \mathbb N\}.$$
\begin{thm}\label{main2}
Let $X$ be a shift space with $\mathcal{L}=\mathcal{L}(X)$. Let $f$ be a positive continuous function defined on $X$. Suppose that $\mathcal{G}\subset \mathcal{L}$ has $(W)$-specification and  $\mathcal{L}$ is edit approachable by $\mathcal{G}$.
 The Hasudorff dimension of $R(f)$ is the unique solution $s$ of the following pressure equation
$$P(-s(f+1) )=0,$$
where $P(\bullet)$ denotes the topological pressure.
\end{thm}
This paper is organized as follows: In section 2, we give our definitions and some key propositions. In section 3, we give the proof of Theorem\ref{main1}. In section 4, we give the proof of Theorem\ref{main2}. In section 4, we give some applications in $\beta$-shifts and $S$-gap shifts.

\section{Preliminaries}
\subsection{non-uniform structure}
In this paper, we consider the symbolic space. Let $p\geq 2$ be an integer and $\mathcal{A}=\{1,\cdots,p\}$. Let
$$\mathcal{A}^{\mathbb N}=\{(w_{i})_{i=1}^{\infty}:~w_{i}\in \mathcal{A}~\text{for} ~i\geq 1\}.$$
Then $\mathcal{A}^{\mathbb N}$ is compact endowed with the product discrete topology.
And we can define the metric of $\mathcal{A}^{\mathbb N}$ as follows, for any $u,v\in \mathcal{A}^{\mathbb N}$, define
\begin{align*}
d(u,v):= e^{-|u\wedge v|},
\end{align*}
where $|u\wedge v|$ denote the maximal length $n$ such that $u_{1}=v_{1}, u_{2}=v_{2},\cdots, u_{n}=v_{n}$.
We say that $(X,\sigma)$ is a subshift over $\mathcal{A}$ if $X$ is a compact subset of $\mathcal{A}^{\mathbb N}$, and $\sigma(X)\subset X$, where $\sigma$ is the left shift map on $\mathcal{A}^{\mathbb N}$, and
$$\sigma((w_{i})_{i=1}^{\infty})=(w_{i+1})_{i=1}^{\infty},~~\forall~~ (w_{i})_{i=1}^{\infty}\in \mathcal{A}^{\mathbb N}.$$
In particular, $(X,\sigma)$ is called the full shift over $\mathcal{A}$ if $X=\mathcal{A}^{\mathbb N}.$ For $n\in \mathbb N$ and $w\in \mathcal{A}^{\mathbb N}$, we write
$$[w]=\{(w_{i})_{i=1}^{\infty}\in \mathcal{A}^{\mathbb N}: ~~~w_{1}\cdots w_{n}=w\},$$
 and call it an $n$-th word in $\mathcal{A}^{\mathbb N}$ and denote all the $n$-th word by $\mathcal{A}^{n}$. The language of $X$, denoted by $\mathcal{L}=\mathcal{L}(X)$, is the set of finite words that appear in some $x\in X$ that is
$$\mathcal{L}(X)=\{w\in \mathcal{A}^{*}: [w]\neq \emptyset\},$$
where $\mathcal{A}^{*}=\cup_{n\geq 0}\mathcal{A}^{n}$. Given $w\in \mathcal{L}$, let $|w|$ denote the length of $w$. For any collection $\mathcal{D}\subset \mathcal{L}$, let $\mathcal{D}_{n}$ denote $\{w\in \mathcal{D}: |w|=n\}.$ Thus, $\mathcal{L}_{n}$ is the set of all words of length $n$ that appear in sequences belonging to $X$. Given words $u,v$ we use juxtaposition $uv$ to denote the word obtained by concatenation.

\begin{defn}\rm{\cite{CliTho}}
Given a shift space $X$ and its language $\mathcal{L}$, consider a subset $\mathcal{G}\subset \mathcal{L}$. Given $\tau\in \mathbb N$, we say that $\mathcal{G}$ has $(W)$-specification with gap length $\tau$ if for every $v, w\in \mathcal{G}$ there is $u\in \mathcal{L}$ such that $vuw\in \mathcal{G}$ and $|u|\leq \tau.$
\end{defn}
\begin{defn}\rm{\cite{CliTho}}
Define an edit of a word $w=w_{1}\cdots w_{n}\in \mathcal{L}$ to be a transformation of $w$ by one of the following actions, where $w^{j}\in \mathcal{L}$ are arbitrary words and $a,a^{'}\in \mathcal{A}$ are arbitrary symbols.

(1) substitution: $w=u^{1}au^{2}\mapsto w^{'}=u^{1}a^{'}u^{2}.$

(2) Insertion: $w=u^{1}u^{2}\mapsto w^{'}=u^{1}a^{'}u^{2}.$

(3) Deletion: $w=u^{1}au^{2}\mapsto w^{'}=u^{1}u^{2}.$
\end{defn}
\noindent Given $v,w\in \mathcal{L}$, define the edit distance between $v$ and $w$ to be the minimum number of edits required to transform the word $v$ into the word $w$, we will denote this by $\hat{d}(v,w).$

The following proposition about describe the  size of balls in the edit metric.
\begin{prop}\rm{\cite{CliTho}}\label{lemmaball}
There is $C>0$ such that given $n\in \mathbb N, w\in \mathcal{L}_{n}$, and $\delta>0$, we have
$$\sharp\{v\in \mathcal{L}:~\hat{d}(v,w)\leq \delta n\}\leq Cn^{C}(e^{C\delta}e^{-\delta\log\delta})^{n}.$$
\end{prop}

  Now we introduce the key definition, which requires that any word in $\mathcal{L}$ can be transformed into a word in $\mathcal{G}$ with a relatively small number of edits.

\begin{defn}\rm{\cite{CliTho}}
Say that a non-decreasing function $g:\mathbb N\to \mathbb N$ is a mistake function if $\frac{g(n)}{n}$ converges to $0$. We say that $\mathcal{L}$ is edit approachable by $\mathcal{G}$, where $\mathcal{G}\subset \mathcal{L}$, if there is a mistake function $g$ such that for every $w\in \mathcal{L}$, there exists $v\in \mathcal{G}$ with $\hat{d}(v,w)\leq g(|w|)$.
\end{defn}

\noindent We can get the following lemma, by applying[\cite{CliTho}, Proposition 4.2 and Lemma 4.3].
\begin{prop}\label{lemma2.2}
If $\mathcal{G}$ has $(W)$-specification, then there exist $\mathcal{F}\subset \mathcal{L}$, which has free concatenation property (if for all $u,w\in \mathcal{F}$, we have $uw\in\mathcal{F}$) and $\mathcal{L}$ is edit approachable by $\mathcal{F}.$
\end{prop}
\begin{rem} We do not have $\mathcal{F}_{n}\neq \emptyset,$ for each $n\in \mathbb N.$
\end{rem}
 To estimate the lower bound, we need the following distribution theorem.
\begin{thm} {\rm \cite{Pes}}\label{thg}
Let $E$ be a Borel measurable set in $X$ and $\mu$ be a Borel measure with $\mu(E)>0$. Assume that there exist two positive  constants $c,\eta$ such that, for any set $U$ with diameter ${\rm diam} U<\eta$, $\mu(U)\leq c{\rm diam}(U)^{s},$ then
$${\rm dim }_{H}E\geq s.$$
\end{thm}

\subsection{Topological pressure}
Given a collection $\mathcal{D}\subset \mathcal{L}$, the entropy of $\mathcal{D}$ is $$h(\mathcal{D}):=\limsup\limits_{n\to\infty}\dfrac{1}{n}\log \sharp \mathcal{D}_{n},$$
where $\mathcal{D}_{n}=\{w\in \mathcal{D}:~|w|=n\}.$   For a fixed potential function $\varphi\in C(X)$, the pressure of $\mathcal{D}\subset \mathcal{L}$ is
$$P(\mathcal{D},\varphi):=\limsup\limits_{n\to\infty} \frac{1}{n}\log \Lambda_{n}(\mathcal{D},\varphi),$$
where $\Lambda_{n}(\mathcal{D},\varphi)=\sum_{w\in \mathcal{D}_{n}}e^{\sup_{x\in [w]}S_{n}\varphi(x)}$
and $S_{n}\varphi(x)=\sum_{k=0}^{n-1}\varphi(\sigma^{k}x).$ We write $P(\varphi):=P(\mathcal{L},\varphi)$.
\begin{prop}\rm{\cite{CliTho}}\label{prop2.3}
If $\mathcal{L}$ is edit approachable by $\mathcal{G}$, then $P(\mathcal{G}, \varphi)=P(\varphi)$ for every $\varphi\in C(X).$
\end{prop}
In the following, we set
 $\mathcal{N}(\mathcal{F}):=\{n\in \mathbb N, \mathcal{F}_{n}\neq \emptyset\}.$

\begin{defn}
Let $g=-(f+1)\in C(X), f>0$  and $\mathcal{L}$ is edit approachable by $\mathcal{F}.$

(1) For any $n\geq 1$, define $s_{n}(X)$ to be the unique solution of the equation
$$\sum_{w\in \mathcal{L}_{n} }\big(e^{\sup_{x\in [w] }S_{n}g(x)}\big)^{s}=1.$$

(2) For any $n\geq 1$  and $n\in \mathcal{N}(\mathcal{F})$, define $ s_{n}(\mathcal{F})$ to be the unique solution of the equation
$$\sum_{w\in \mathcal{F}_{n} }\big(e^{\sup_{x\in [w] }S_{n}g(x)}\big)^{ s }=1.$$

\end{defn}
\begin{rem}
Since $f+1>1$ is a continuous function on $X$, the above definitions is well defined.
\end{rem}

\begin{prop}\label{prop2.4}
 Assume   $s(\mathcal{F})$ to be the solution of the pressure equations
$$ P(\mathcal{F},-s(f+1) )=0.$$
   For the increasing sequence $\{n_{j}\}_{j\geq 1}= \mathcal{N}(\mathcal{F}),$  we have
\begin{align*}
\begin{split}
 \lim\limits_{j\to\infty}s_{n_{j}}(\mathcal{F})=s(\mathcal{F}).
\end{split}
\end{align*}
 \end{prop}
\begin{proof}
By virtue of the definition of $P(\mathcal{F},-s(f+1) )=0 $, it is easy to see that the solution of $ P(\mathcal{F},-s(f+1) )=0 $ is unique and
pressure function $f\mapsto P_{top}(\mathcal{F},f )$ is continuous. Accordingly,
we claim that $s_{n}(\mathcal{F})$ is bounded for each $n\in \mathcal{N}(\mathcal{F})$. This is because
$$\sum_{\mathcal{F}_{n}}e^{-ns_{n}(\mathcal{F})||f+1||_{max}}\leq 1\leq \sum_{\mathcal{F}_{n}}e^{-ns_{n}(\mathcal{F})||f+1||_{min}}.$$
Then
$$0<\frac{1}{||f+1||_{max}}\frac{\log \sharp \mathcal{F}_{n}}{n}\leq s_{n}(\mathcal{F})\leq \frac{1}{||f+1||_{min}}\frac{\log \sharp \mathcal{F}_{n}}{n}.$$
With the fact that $\limsup\limits_{n\to\infty}\frac{\log \sharp \mathcal{F}_{n}}{n}\leq \lim\limits_{n\to\infty}\frac{\log \sharp \mathcal{L}_{n}}{n}=h_{top}(X)$, we can know it is bounded.
 Moreover, by the continuity of pressure function $f\mapsto P_{top}(\mathcal{F},f )$, it is easy to get $\liminf\limits_{ j \to\infty}s_{n_{j}}(\mathcal{F})$ and $\limsup\limits_{ j \to\infty}s_{n_{j}}(\mathcal{F})$ are the solution of  $P(\mathcal{F},-s(f+1) )=0 $.
Hence,
$$\lim\limits_{j\to\infty}s_{n_{j}}(\mathcal{F})=s(\mathcal{F}).$$
\end{proof}

\begin{cro}If $\mathcal{L}$ is edit approachable by $\mathcal{F}$. Assume $s (X)$ and $s(\mathcal{F})$ to be, respectively, the solution of the pressure equations
 $P(-s(f+1) )=0,~ P(\mathcal{F},-s(f+1) )=0,$  then
 $$s(X)= s (\mathcal{F}).$$
\end{cro}

 \section{Proof of Theorem\ref{main1}}
 Firstly, we consider that {\bf C1)}.
By $\liminf\limits_{n\to \infty}\psi(n)>0$. Namely, there exists $\epsilon_{0}>0$, and $N>0 $ such that for any $n\geq N$,  $ \psi(n) \geq \epsilon_{0}.$
Clearly, we have $b=0.$ We only need to show that ${\rm dim}_{H}R(\psi)\geq h.$ By the upper semi-continuity of the entropy map, we can choose ergodic measure $\mu$ such that $h_{\mu}(\sigma)=h.$
By Poincar\'{e}  recurrence theorem, we have $\mu(R(\psi))=1$.
Hence,
$${\rm dim }_{H}R(\psi)\geq {\rm dim}_{H}\mu=\lim\limits_{n\to \infty}\frac{-\log \mu(I_{n}(x))}{n}=h_{\mu}(\sigma).$$

\noindent Secondly, the proof of {\bf C2)} is divided into two parts.
 \subsection{Upper bound}

 The upper bound can be obtained by considering the natural covering system. Evidently,
 $$R(\psi)=\bigcap_{N=1}^{\infty}\bigcup_{n=N}^{\infty}\bigcup_{(w_{1},w_{2},\cdots ,w_{n})\in\mathcal{L}_{n}}J(w_{1},w_{2},\cdots, w_{n}),$$
 where
 $$J(w_{1},w_{2},\cdots, w_{n}):=\{x\in X:x\in [w_{1}w_{2}\cdots w_{n}], d(\sigma^{n}x,x)<\psi(n)\}.$$
 Obviously, we can estimate the diameter of $J(w_{1},w_{2},\cdots, w_{n})$ by
 $${\rm diam}(J(w_{1},w_{2},\cdots, w_{n}))\leq e^{-n}\psi(n) .$$
 As a result, for any $s>\frac{h}{1+b}$, and without lost generality we can assume that $ s=\frac{h(1+\delta)}{1+b}, $ for some $\delta>0.$
 Also, by the definition of the topological entropy of $X$ and the definition of $b$, we can choose $\epsilon$ such that $(\frac{h(1+\delta)}{1+b}+1)\epsilon<\frac{h\delta}{2}$, and then
 we have $$\sharp\mathcal{L}_{n}(X)\leq e^{n(h+\epsilon)}~\text{and}~\psi(n)\leq e^{-n(b-\epsilon)}$$ for $n$ large enough.
Hence,
\begin{align*}
\begin{split}
H^{s}(R(\psi))&\leq \liminf\limits_{N\to\infty}\sum_{n=N}^{\infty}\sum_{ (w_{1}w_{2}\cdots w_{n})\in\mathcal{L}_{n}}{\rm diam}(J(w_{1},w_{2},\cdots, w_{n}))^{s}\\
&\leq \liminf\limits_{N\to\infty}\sum_{n=N}^{\infty}e^{n(h+\epsilon)}(e^{-n}\psi(n))^{s}\\
&\leq \liminf\limits_{N\to\infty}\sum_{n=N}^{\infty}e^{-\frac{nh\delta}{2}}.
\end{split}
\end{align*}
 Furthermore, $$H^{s}(R(\psi))< \infty.$$
 This implies
 $${\rm dim}_{H}(R(\psi))\leq \dfrac{h}{1+b}.$$

  \subsection{Lower bound}

\paragraph{Construction of the Moran set}

Fix $\eta>0$, by Proposition \ref{prop2.3}, we can choose $M$ large enough such that
  $$\log \sharp \mathcal{F}_{M}\geq (1-\eta)Mh.$$
 Choose a largely sparse subsequence $\{n_{k}\}_{k\geq 1}$ of $\mathbb N$, such that
\begin{align}\label{eq2}
 \liminf\limits_{n\to\infty}\frac{-\log \psi(n)}{n}=\lim\limits_{k\to\infty}\frac{-\log \psi(n_{k})}{n_{k}},~~~~~~~~~~~~~\frac{n_{k}}{k}\geq \max\big\{\sum_{j=1}^{k-1}n_{j}, ~-\log \psi(n_{k-1}) \big\}.
\end{align}
For $k= 1$,
define $l_{1}, i_{1}$ such that $n_{1}= l_{1} M+i_{1}$, $0\leq i_{1}<M.$
We define $\hat{n}_{1}= l_{1} M$, and
  the integer $\hat{t}_{1}$,
 \begin{align*}
e^{-\hat{t}_{1}}< \psi(\hat{n}_{1})\leq e^{- \hat{t}_{1}+1},
\end{align*}
and then we choose $t_{1}$ by modifying $\hat{t}_{1}$ such that
$\hat{t}_{1}+M\geq t_{1}\geq \hat{t}_{1}$ and $ M\mid t_{1} $.
As a consequences, we obtain
\begin{align*}
e^{- t_{1}}< \psi(\hat{n}_{1})\leq e^{-  t_{1}+M+1}.
\end{align*}
With $\psi$ is  nonincreasing, we have
$$e^{-  t_{1}+M+1}\geq \psi(\hat{n}_{1})\geq \psi( n_{1}).$$
 And then define the rational number $r_{1}$ such that
 $$~~ \hat{n}_{1}r_{1}=\hat{n}_{1}+t_{1}.$$
For $k\geq 2$,
define $l_{k}, i_{k}$ such that
 $ n_{k}-(\hat{n}_{k-1}+t_{k-1})=  l_{k}   M +i_{k} $, $0\leq i_{k}<M.$
 And then we define $\hat{n}_{k}:=\hat{n}_{k-1}+t_{k-1}+ l_{k}   M.$
 Define the integer $\hat{t}_{k}$,
 \begin{align*}
e^{-\hat{t}_{k}}< \psi(\hat{n}_{k})\leq e^{- \hat{t}_{k}+1},
\end{align*}
and then we choose $t_{k}$  satisfy $M\mid t_{k} $ and
$\hat{t}_{k}+ M\geq t_{k}\geq \hat{t}_{k}$.
As a consequences, we have
\begin{align}\label{eq1}
e^{- t_{k}}< \psi(\hat{n}_{k})\leq e^{-  t_{k}+ M+1}.
\end{align}
With $\psi(n)$ is nonincreasing, we have
\begin{align}\label{eqjiu}
e^{-t_{k}+ M+1}\geq \psi(\hat{n}_{k})\geq \psi(n_{k}).
\end{align}
And then define the rational number $r_{k}$ such that
 $$~~ \hat{n}_{k}r_{k}=\hat{n}_{k}+t_{k}.$$
 From the definitions, we can see that
 \begin{align}\label{eqn}
 n_{k}- M\leq \hat{n}_{k}\leq n_{k}.
 \end{align}
We are now in the place to construct a Moran subset of $R(\psi)$ as follows.
As we realize the events $d(\sigma^{n}x,x)<\psi(n)$ for infinitely many times along the subsequence $\{\hat{n}_{k}\}_{k\geq 1}.$

 \subparagraph{ Level 1 of the Moran set.}
Recall the definition of $t_{1}$ and $r_{1}$
$$\mathbb F(1)=\bigcup [(\underbrace{w_{1} \cdots     w_{ M }}_{ M }     \cdots  \underbrace{w_{(l_{1}-1) M +1} \cdots  w_{l_{1} M }}_{ M }    )^{r_{1}}],$$
where the union is taken over all blocks $(w_{l M +1},\cdots, w_{(l+1) M })\in \mathcal{F}_{M}$ for each $0\leq l\leq l_{1}-1  $. Since $\mathcal{F}$ has free concatenation property, the concatenation is admissible.
From the construction, we have that for any word $I\in \mathbb F(1)$ and $x\in I$, the prefix of $\sigma^{\hat{n}_{1}}x$ and $x$ coincide at the first $t_{1}$ digits.

\subparagraph{ Level 2 of the Moran set.} The second level sets is composed of collection of words of each word $J_{1}\in \mathbb F(1)$:
$$\mathbb F(2)=\bigcup_{J_{1}\in \mathbb F(1)} \mathbb F(2,J_{1}),$$
where for a fixed $J_{1}\in \mathbb F(1),$ writing $J_{1}=[(w_{1} \cdots  w_{\hat{n}_{1}+t_{1}})] $ and

 \begin{align*}
 \begin{split}
 \mathbb F(2,J_{1})=\bigcup & [(w_{1} \cdots  w_{\hat{n}_{1}+t_{1}} \underbrace{w_{\hat{n}_{1}+t_{1}+1} \cdots  w_{\hat{n}_{1}+t_{1}+ M }}_{ M } \cdots
   \underbrace{w_{\hat{n}_{1}+t_{1}+(l_{2}-1) M +1} \cdots  w_{\hat{n}_{1}+t_{1}+l_{2} M }}_{ M }
    )^{r_{2}}],
 \end{split}
 \end{align*}
where the union is taken over all blocks $(w_{\hat{n}_{1}+t_{1}+l M +1},\cdots, w_{\hat{n}_{1}+t_{1}+(l+1) M })\in \mathcal{F}_{M}$ for each $0\leq l\leq l_{2}-1  $. Since $\mathcal{F}$ has free concatenation property, the concatenation is admissible.
From the construction, we have that for any word $I\in \mathbb F(2)$ and $x\in I$, the prefix of $\sigma^{\hat{n}_{2}}x$ and $x$ coincide at the first $t_{2}$ digits.

  \subparagraph{ From level k to level k+1}
 Provided that $\mathbb{F}(k)$ has been defined, we define  $\mathbb F(k+1)$ as follows:
 $$\mathbb{F}(k+1)=\bigcup_{J_{k}\in \mathbb{F}(k)}\mathbb{F}(k+1,J_{k}),$$
 where for any  $J_{k}=[(w_{1} \cdots  w_{\hat{n}_{k}+t_{k}})]\in \mathbb{F}(k),$

  \begin{align*}
 \begin{split}
 &\mathbb F(k+1 , J_{k})\\&=\bigcup  [(w_{1} \cdots  w_{\hat{n}_{k}+t_{k}} \underbrace{w_{\hat{n}_{k}+t_{k}+1} \cdots  w_{\hat{n}_{k}+t_{k}+ M }}_{ M } \cdots
    \underbrace{w_{\hat{n}_{k}+t_{k}+(l_{k+1}-1) M +1} \cdots    w_{\hat{n}_{k}+t_{k}+l_{k+1} M }}_{ M } )^{r_{k+1}}],
 \end{split}
 \end{align*}
where the union is taken over all blocks $(w_{\hat{n}_{k}+t_{k}+l M +1},\cdots, w_{\hat{n}_{k}+t_{k}+(l+1) M })\in \mathcal{F}_{M}$ for each $0\leq l\leq l_{k+1}-1 $. Since $\mathcal{F}$ has free concatenation property, the concatenation is admissible.
From the construction, we have that for any word $I\in \mathbb F(k+1)$ and $x\in I$, the prefix of $\sigma^{\hat{n}_{k+1}}x$ and $x$ coincide at the first $t_{k+1}$ digits.
\subparagraph{ The Moran set} We obtain a nested sequence $\{\mathbb{F}(k)\}_{k\geq 1}$ composed of word. And then the Moran set is obtained as
 $$\mathbb F_{\infty}=\bigcap_{k=1}^{\infty}\mathbb F(k).$$
 By the above constructions, we get
 $$\mathbb F_{\infty}\subset R(\psi).$$

 \paragraph{Supporting measure}
 Now we construct a probability measure $\mu$ on $\mathbb F_{\infty}$.
 For any $J_{k}\in \mathbb F(k)$, letting $J_{k-1}\in \mathbb F(k-1)$ be its mother word, i.e.,$J_{k}\in \mathbb F(k,J_{k-1})$, the measure of $J_{k}$ is defined as
 \begin{align*}\begin{split}
 \mu(J_{k})&:=\dfrac{1}{ \sharp  \mathbb F(k,J_{k-1}) }\mu(J_{k-1})\\
 &=   \prod_{j=1}^{k}\dfrac{1}{(\sharp\mathcal{F}_{M})^{l_{j} }  }.
 \end{split}
 \end{align*}
 This means that the measure of any mother word is evenly distributed among her offsprings.
 For any $n\geq 1$, and $n$ long word $I_{n}=[w_{1}\cdots w_{n}] $ with $I_{n}\cap\mathbb F_{\infty}\neq \emptyset,$ let $k\geq 2$ be the integer such that $\hat{n}_{k-1}+t_{k-1}<n\leq \hat{n}_{k}+t_{k}$. We just set
 $$\mu([w_{1}\cdots w_{n}])=\sum_{J_{k}\subset I_{n}}\mu(J_{k}),$$
 where the summation is taken over all words $J_{k}\in \mathbb F(k)$ contained in $I_{n}$. In fact, we have the following expression for the measure of a word.

(1) If $\hat{n}_{k-1}\leq n\leq \hat{n}_{k-1}+t_{k-1}$,
 $$\mu(I_{n})=\mu(I_{\hat{n}_{k-1}+t_{k-1}}).$$

 (2) If $\hat{n}_{k-1}+t_{k-1}<n<\hat{n}_{k},$ assume $n=\hat{n}_{k-1}+t_{k-1}+lM+i.$

 For $i=0$ and $0\leq l\leq l_{k}-1$,
   $$\mu(I_{n})=\mu(I_{\hat{n}_{k-1}+t_{k-1}})\frac{1}{(\sharp\mathcal{F}_{M}\big)^{l}}.$$

For $i\neq 0$ and $0\leq l\leq l_{k}-1$,
$$\mu(I_{n})\leq \mu(I_{\hat{n}_{k-1}+t_{k-1}+lM}).$$
\paragraph{H\"{o}lder exponent of the measure}

Connect (\ref{eqjiu}) with the fact that $b=\lim\limits_{ k \to\infty}\frac{-\log \psi(n_{k})}{n_{k}}$,
we obtain $$\lim\limits_{ k \to \infty}\frac{t_{k}}{n_{k}}\leq b.$$
Accordingly, from (\ref{eqn}), we have
 $$\lim\limits_{ k \to \infty}\frac{t_{k}}{\hat{n}_{k}}\leq b.$$
Furthermore, by virtue of (\ref{eq2}), (\ref{eqjiu}) and (\ref{eqn}), there exists $k_{0}$ such that
$ k \geq k_{0}$ satisfy
$$\dfrac{\hat{n}_{k}-\hat{n}_{k-1}-t_{k-1}}{\hat{n}_{k}+t_{k}}\geq \frac{1- \eta}{1+b}.$$
Firstly, we consider $J_{k}$,
\begin{align*}
\begin{split}
\dfrac{-\log \mu(J_{k})}{\hat{n}_{k}+t_{k}}&=\dfrac{\sum_{j=1}^{k} l_{j} \log \sharp\mathcal{F}_{M}  }{\hat{n}_{k}+t_{k}}
\\
&\geq \dfrac{\hat{n}_{k}-\hat{n}_{k-1}-t_{k-1}}{\hat{n}_{k}+t_{k}}\times\dfrac{\log \sharp\mathcal{F}_{M} }{ M }\\
&\geq \dfrac{h(1- \eta)^{2}}{1+b}.
\end{split}
\end{align*}
Hence,
\begin{align*}
\begin{split}
  \mu(J_{k})& \leq e^{-(\hat{n}_{k}+t_{k})\dfrac{h(1- \eta)^{2}}{1+b}}.\\
 \end{split}
\end{align*}
(1) If $\hat{n}_{k-1}\leq n\leq \hat{n}_{k-1}+t_{k-1}$,
 $$\mu(I_{n})=\mu(I_{\hat{n}_{k-1}+t_{k-1}}). $$
And then
$$\dfrac{-\log \mu(I_{n})}{n}\geq \dfrac{-\log \mu(I_{\hat{n}_{k-1}+t_{k-1}})}{\hat{n}_{k-1}+t_{k-1}}\geq \dfrac{h(1- \eta)^{2}}{1+b}.$$
This implies that
$$\mu(I_{n})\leq {\rm diam}(I_{n})^{\dfrac{h(1- \eta)^{2}}{1+b}}.$$
 (2) If $\hat{n}_{k-1}+t_{k-1}< n<\hat{n}_{k},$ denote $n=\hat{n}_{k-1}+t_{k-1}+lM+i.$

For $i=0$ and $0\leq l\leq l_{k}-1$,
   \begin{align*}
   \begin{split}
   \mu(I_{n})&=\mu(I_{\hat{n}_{k-1}+t_{k-1}})\frac{1}{ \big(\sharp\mathcal{F}_{M}\big)^{l}}\\
   &\leq  \mu(I_{\hat{n}_{k-1}+t_{k-1}}) e^{-l(1-\eta)Mh}.
   \end{split}
   \end{align*}
And then
\begin{align*}
\begin{split}
\dfrac{-\log \mu(I_{n})}{n}&\geq \dfrac{-\log \mu(I_{\hat{n}_{k-1}+t_{k-1}})+l(1- \eta) M h}{\hat{n}_{k-1}+t_{k-1}+l M  }\\&
= \min\Big\{ \dfrac{-\log \mu(I_{\hat{n}_{k-1}+t_{k-1}}) }{\hat{n}_{k-1}+t_{k-1}},  (1- \eta)h \Big\}\\&
\geq \dfrac{h(1- \eta)^{2}}{1+b}.
\end{split}
\end{align*}
This implies that
$$\mu(I_{n})\leq {\rm diam}(I_{n})^{\dfrac{h(1- \eta)^{2}}{1+b}}.$$

For $i\neq 0$ and $0\leq l\leq l_{k}-1$,
$$\mu(I_{n})\leq \mu(I_{\hat{n}_{k-1}+t_{k-1}+lM}).$$
From $${\rm diam}(I_{n})\geq e^{-  M } {\rm diam}(I_{\hat{n}_{k-1}+t_{k-1}+lM}),$$
we have
$$\mu(I_{n})\leq e^{  M \dfrac{h(1- \eta)^{2}}{1+b}}{\rm diam}(I_{n})^{\dfrac{h(1- \eta)^{2}}{1+b}}.$$
Finally, by Theorem \ref{thg} and $\eta$ can be arbitrary small, we finish the proof of {\bf C2)}.


 \section{Proof of Theorem\ref{main2}}
 Naturally, the proof is divided into two parts.
 \subsection{Upper bound}
 The proof is similar to the proof of the upper bound of Theorem\ref{main1}.
 Clearly,

  $$R(f)=\bigcap_{N=1}^{\infty}\bigcup_{n=N}^{\infty}\bigcup_{(w_{1} w_{2} \cdots w_{n})\in\mathcal{L}_{n}}J(w_{1},w_{2},\cdots, w_{n}),$$
 where
 $$J(w_{1},w_{2},\cdots, w_{n}):=\{x\in X:x\in [w_{1}w_{2}\cdots w_{n}], d(\sigma^{n}x,x)<e^{-S_{n}f(x)}\}.$$
For each $(w_{1}w_{2}\cdots w_{n})$, we can choose $y$  such that
$$S_{n}f(y)=\sup_{x\in [w_{1}w_{2}\cdots w_{n}]}S_{n}f(x). $$ Thus, by the continuity of $f$, for each $\delta>0$ and $n$ large enough, we obtain
 $$J(w_{1},w_{2},\cdots, w_{n})\subset \{x\in X:x\in [w_{1}w_{2}\cdots w_{n}], d(\sigma^{n}x,x)<e^{-S_{n}f(y)}e^{n\delta}\},$$
 where  $ S_{n}f(y)=\inf_{x\in [x_{1}\cdots x_{n}]}S_{n}f(x) .$
Thus,
$${\rm diam}(J(w_{1},w_{2},\cdots, w_{n}))\leq e^{-S_{n}f(y)+n\delta-n}. $$
We define $s(\delta)$ be the solution of $P (   (s    (-1-f+\delta) )=0$, and by the continuity of the pressure function $f\mapsto P(f)$  and the boundedness of $s(\delta)$, we obtain  $\lim\limits_{\delta\to 0^{+}}s(\delta)=s(X).$
 At the same time,  we denote $ P:=P (  (s  (\delta)+\delta) (-1-f+\delta) )<0,$
 then there exists $\epsilon(\delta)>0$, such that
 $$\sum_{w\in \mathcal{L}_{n}(X )}\big(e^{-S_{n}(f(x)+1-\delta)}\big)^{ s  (\delta)+\delta }\leq e^{ -n\epsilon(\delta)}, $$
 for $n$ large enough.
Moreover,
\begin{align*}
\begin{split}
H^{ s  (\delta)+\delta}(R(f))&\leq \liminf\limits_{N\to\infty}\sum_{n=N}^{\infty}\sum_{ (w_{1}w_{2}\cdots w_{n})\in\mathcal{L}_{n}}{\rm diam}(J(w_{1},w_{2},\cdots, w_{n}))^{ s  (\delta)+\delta}\\
&\leq \liminf\limits_{N\to\infty}\sum_{n=N}^{\infty} e^{  -n\epsilon(\delta)}<\infty.\\
 \end{split}
\end{align*}
This implies that
$${\rm dim}R(f)\leq s(\delta)+\delta.$$
With $\delta$ is arbitrary small, we finished the proof.

\subsection{Lower bound}
By the continuity of $f$, choose
$ y\in [x_{1}\cdots x_{n}] $ such that $ S_{n}f(y)=\inf_{x\in [x_{1}\cdots x_{n}]}S_{n}f(x) ,$ i.e., $y$ depends on $n,x.$
It suffices to show that the result holds for the set
$$\Big\{x\in X:~d(x,\sigma^{n}x)< e^{-S_{n}f(y)} ~\text{with}~\text{for infinitely many }~n\in \mathbb N\Big\}.$$
Fix $\eta>0$, by Proposition\ref{prop2.4}, we can choose $M\in \mathcal{N}(\mathcal{F})$ and moreover, for any $n, m\geq M$ with $n, m\in \mathcal{N}(\mathcal{F})$,    satisfying
\begin{align}\label{equ}
\begin{split}
&\sup\{|f(x)- f(y)|:x, y\in X, d(x,y)\leq e^{-M}\}\leq \frac{\eta}{4},\\&~~|s_{n}(\mathcal{F})-s(X)|<\eta ~\text{and}~~|s_{n}(\mathcal{F})-s_{m}(\mathcal{F})|<\eta.
\end{split}
\end{align}
Since $\mathcal{F}$ has free concatenation property, we can see that
$M, 2M,3M,\cdots  \in \mathcal{N}(\mathcal{F}).$
\subsubsection{Construction of the Moran set}
In the following, for any $ [w_{1} \cdots w_{n }]$, we set $y\in [w_{1} \cdots w_{n }]$ satisfies
$$S_{n}f(y):=\inf_{x\in [w_{1} \cdots w_{n }]}S_{n}f(x).$$
For $k=1  $,
choose $m_{1}=M$  and define $n_{1}:=m_{1}$.
For any $(w_{1}\cdots w_{n_{1}})\in \mathcal{F}$, define $\hat{t}=\hat{t}(w_{1}\cdots w_{n_{1}})$ to be the integer,
\begin{align*}
e^{-\hat{t}_{1}}<e^{-S_{n_{1}}f(y)}\leq e^{- \hat{t}_{1}+1},
\end{align*}
where $S_{n_{1}}f(y)=\inf_{x\in [w_{1}\cdots w_{n_{1}}]}S_{n_{1}}f(x).$
Moreover, we choose $t_{1}$ by modifying $\hat{t}_{1}$ such that
$\hat{t}_{1}+M\geq t_{1}\geq \hat{t}_{1}$ and $M\mid t_{1}. $ So
\begin{align*}
e^{- t_{1}}< e^{-S_{n_{1}}f(y)}\leq e^{-  t_{1}+M+1}.
\end{align*}
And then define $r_{1}$ such that
 $$~~ n_{1}r_{1}=n_{1}+t_{1}.$$
For $k\geq 2$,
define $m_{k}$ satisfy $ M \mid m_{k} $,  we can choose $m_{k}$
large enough such that
\begin{align}\label{eq3}
 (n_{k-1}+t_{k-1}) ||f||\leq \frac{m_{k}\eta}{2}, ~~\text{and}~\frac{m_{k}}{k}\geq m_{1}+\cdots+m_{k-2}+  m_{k-1},
\end{align}
and then define $ n_{k} =m_{k}+n_{k-1}+t_{k-1}$.
Define the integer $\hat{t}_{k}$,
 \begin{align*}
e^{-\hat{t}_{k}}< e^{-S_{n_{k}}f(y)}\leq e^{- \hat{t}_{k}+1},
\end{align*}
and then we choose $t_{k}$  satisfy $M\mid t_{k} $  and satisfy
$\hat{t}_{k}+M\geq t_{k}\geq \hat{t}_{k}$.
As a consequences, we have
\begin{align}\label{eq4a}
e^{- t_{k}}< e^{-S_{n_{k}}f(y)}\leq e^{-  t_{k}+M+1}.
\end{align}
And then define $r_{k}$ such that
 $$~~ n_{k}r_{k}=n_{k}+t_{k}.$$
We are now in the place to construct a Moran subset of $R(f)$ as follows.

{\bf Level 1 of the Moran set.}
Recall the definition of $t_{1}$ and $r_{1}$
$$\mathbb F(1)=\bigcup [( w_{1}^{1}  \cdots     w_{m_{1}}^{1}        )^{r_{1}}],$$
where the union is taken over all blocks $ ( w_{1}^{1}  \cdots     w_{m_{1}}^{1}        )\in  \mathcal{F}_{m_{1}} $,
 with $r_{1}m_{1}=m_{1}+t_{1}.$
Since $\mathcal{F}$ has free concatenation property, the concatenation is admissible.
From the construction, we have that for any word $I\in \mathbb F(1)$ and $x\in I$, the prefix of $\sigma^{n_{1}}x$ and $x$ coincide at the first $t_{1}$ blocks.

 {\bf Level 2 of the Moran set.} The second level is composed of collection of words of each word $J_{1}\in \mathbb F(1)$:
$$\mathbb F(2)=\bigcup_{J_{1}\in \mathbb F(1)} \mathbb F(2,J_{1}),$$
where for a fixed $J_{1}\in \mathbb F(1),$ writing $J_{1}=[(w_{1}  \cdots  w_{m_{1}+t_{1}} )].$

 \begin{align*}
 \begin{split}
 \mathbb F(2,J_{1})=\bigcup & [(w_{1}   \cdots  w_{m_{1}+t_{1}}   w_{1}^{2}  \cdots w_{m_{2}}^{2})^{r_{2}}]
 \end{split}
 \end{align*}
where the union is taken over all blocks $ (w_{1}^{2}  \cdots w_{m_{2}}^{2})\in \mathcal{F}_{m_{2}} $,
 with $r_{2}n_{2}=n_{1}+t_{1}+m_{2}+t_{2}$.
Since $\mathcal{F}$ has free concatenation property, the concatenation is admissible.

  {\bf From level k to level k+1}
 Provided that $\mathbb{F}(k)$ has been defined, we define  $\mathbb F(k+1)$ as follows.
 $$\mathbb{F}(k+1)=\bigcup_{J_{k}\in \mathbb{F}(k)}\mathbb{F}(k+1,J_{k}),$$
 where for any  $J_{k}=[(w_{1}  \cdots  w_{t_{k}+n_{k}} )]\in \mathbb{F}(k),$

  \begin{align}
 \begin{split}
 \mathbb F(k+1 , J_{k})=\bigcup & [(w_{1}  \cdots  w_{t_{k}+n_{k}}w_{1}^{k+1} \cdots w_{m_{k+1}}^{k+1})^{r_{k+1}}],
 \end{split}
 \end{align}
where the union is taken over all blocks $ (w_{1}^{k+1},\cdots w_{m_{k+1}}^{k+1})\in \mathcal{F}_{m_{k+1}}  $,
 with $r_{2}n_{k+1}=n_{k}+m_{k+1}+t_{k+1}.$
Since $\mathcal{F}$ has free concatenation property, so the concatenation is admissible.

 {\bf The Moran set} We obtain a nested sequence $\{\mathbb{F}(k)\}_{k\geq 1}$ composed of word. And then the Moran set is obtained as
 $$\mathbb F_{\infty}=\bigcap_{k=1}^{\infty}\mathbb F(k).$$
From the above constructions, we can get
 $$\mathbb F_{\infty}\subset R(f).$$

 \subsubsection{Supporting measure}
 Now we construct a probability measure $\mu$ on $\mathbb F_{\infty}$.
 For any $J_{k}\in \mathbb F (k)$, letting $J_{k-1}\in \mathbb F(k-1)$ be its mother word, i.e.,$J_{k}\in \mathbb F(k,J_{k-1})$, the measure of $J_{k}$ is defined as
 \begin{align*}\begin{split}
 \mu(J_{k})&:=e^{-s_{k} m_{k}    - s_{k}S_{m_{k}}f(y^{k}) }\mu(J_{k-1})\\
 &=  \prod_{j=1}^{k}e^{-s_{j} m_{j}    -  s_{j}S_{m_{j}}f(y^{j}) }
 \end{split}
 \end{align*}
 where $s_{k}:=s_{m_{k}}(\mathcal{F})$ and $S_{m_{j}}f(y^{j})=\inf_{x\in [w^{j}_{1}\cdots w^{j}_{m_{j}}]}S_{m_{j}}f(x)$ with $y^{j}\in [w^{j}_{1}\cdots w^{j}_{m_{j}}]$  .
This means that the measure of any mother word is evenly distributed among her offsprings.
 For any $n\geq 1$, and $n$ long word $I_{n}=[w_{1}\cdots w_{n}] $ with $I_{n}\cap\mathbb F_{\infty}\neq \emptyset,$
 note that $n_{k}$ and $t_{k}$ depend on the digits. That is, given a block $[w_{1}\cdots w_{n}]$ of length $n$, it determines $t_{1}$ if $n\geq n_{1}$. If $t_{k-1}$ can be determined, we then compare $n$ with $n_{k}=n_{k-1}+t_{k-1}+m_{k}$. If $n\geq n_{k}$, it determines $t_{k}$; otherwise, we have $n_{k-1}\leq n<n_{k}+t_{k-1}+m_{k}=n_{k}$, and the block  $[w_{1}\cdots w_{n}]$ determines $n_{1}$ up to $n_{k-1}$.

Now we consider any word with length  $n_{k-1}\leq n<n_{k} $. We just set
 $$\mu([w_{1}\cdots w_{n}])=\sum_{J_{k}\subset I_{n}}\mu(J_{k}).$$
 where the summation is taken over all words $J_{k}\subset \mathbb F(k)$ contained in $I_{n}$.
In fact, we have the following expression for the measure of a word.

 (1) When $n_{k-1}\leq n\leq n_{k-1}+t_{k-1}$,
 $$\mu(I_{n})=\mu(I_{n_{k-1}+t_{k-1}})=  \prod_{j=1}^{k-1}e^{-s_{j} m_{j}    -  s_{j}S_{m_{j}}f(y^{j}) }. $$

 (2) When $n_{k-1}+t_{k-1}<n<n_{k},$

\begin{align*}
\begin{split}
   \mu(I_{n})&= \sum_{J_{k}\subset I_{n}}\mu(J_{k}) \\
   &=\sum_{(w_{n+1}\cdots w_{n_{k}})\in \Xi}\mu(I_{n_{k}+t_{k}}(w_{1}\cdots w_{n}w_{n+1}\cdots w_{n_{k}})),
\end{split}
\end{align*}
where $\Xi$ denote the set of $(w_{n+1}\cdots w_{n_{k}})$ such that $ (w_{1}\cdots w_{n}w_{n+1}\cdots w_{n_{k}} )\in \mathcal{F}_{m_{k}}.$

\subsubsection{H\"{o}lder exponent of the measure}

Firstly, we consider the $J_{k}$, by (\ref{equ}), for $k$ large enough,

 \begin{align*}\begin{split}
 \mu(J_{k})\leq  \prod_{i=1}^{k}\big( e^{-  m_{j}    -  S_{m_{j}}f(y^{j}) } \big)^{s(X)-\eta}.
 \end{split}
 \end{align*}
From inequality (\ref{equ}) and (\ref{eq3}), for any $J_{k}=[w_{1}\cdots w_{n_{k}}],$
\begin{align}\label{eq4}
\begin{split}
&\sum_{j=1}^{k}|S_{m_{j}}f(y^{j})-S_{n_{j}}f(y)|\\&\leq \sum_{j=1}^{k}\Big(|S_{m_{j}}f(y^{j})-S_{m_{j}}f(\sigma^{n_{j-1}+t_{j-1}}y_{0}^{j})|+|S_{m_{j}}f(\sigma^{n_{j-1}+t_{j-1}}y_{0}^{j})-S_{n_{j}}f(y)|\Big)\\&\leq  \sum_{j=1}^{k} (\frac{m_{j}\eta}{4}+\frac{m_{j}\eta}{4}+\frac{m_{j}\eta}{2}) \leq 2m_{k}\eta,
\end{split}
\end{align}
where $y^{j}_{0}\in [w_{1}\cdots w_{n_{j}}]$ for any $1\leq j\leq k$ and any $y\in J_{k}$.
Furthermore,
 \begin{align*}\begin{split}
 \mu(J_{k})&\leq  \prod_{j=1}^{k}\big( e^{-  m_{j}    -  S_{m_{j}}f(y^{j}) } \big)^{s(X)-\eta} \\
 &\leq  \prod_{j=1}^{k}\big( e^{-  m_{j}    -  S_{n_{j}}f(y ) } \big)^{s(X)-\eta}e^{2m_{k}\eta(s(X)-\eta)} \\
 &\leq  \prod_{j=1}^{k}\big(e^{-m_{i}  - t_{j}+M+1}\big)^{s(X)-\eta}e^{2m_{k}\eta(s(X)-\eta)} \\
 &\leq e^{-(n_{k}+t_{k})(s(X)-\eta)+2m_{k}\eta(s(X)-\eta)+2M(s(X)-\eta)}\\
 &\leq  e^{-(n_{k}+t_{k})(s(X)-\eta -2 \eta(s(X)-\eta)) +2M(s(X)-\eta) }\\
 &=C(\eta)e^{-(n_{k}+t_{k})(s(X)-\Delta(\eta))},
 \end{split}
 \end{align*}
 where $C(\eta):=e^{2M(s(X)-\eta)}$ is bounded, and $\Delta(\eta):=\eta +2 \eta(s(X)-\eta)$ satisfy $\Delta(\eta)\to 0(\eta\to 0).$
The second inequality follows from (\ref{eq4}) and the third inequality follows from (\ref{eq4a}).
And then, we can give the following estimate.

 (1)If $n_{k-1}\leq n\leq n_{k-1}+t_{k-1}$, then
\begin{align*}
\begin{split}
\mu(I_{n})&=\mu(I_{n_{k-1}+t_{k-1}})\leq C(\eta){\rm diam}(J_{k-1})^{(s(X)-\Delta(\eta) )}\\
&=C(\eta)e^{-n(s(X)-\Delta(\eta))}e^{-(n_{k-1}+t_{k-1}-n)(s(X)-\Delta(\eta))}\\
&\leq C(\eta){\rm diam}(I_{n})^{ (s(X)-\Delta(\eta))}.
\end{split}
\end{align*}

 (2)If $n_{k-1}+t_{k-1}<n< n_{k},$ devote $n=n_{k-1}+t_{k-1}+l, $ then

\begin{align*}
\begin{split}
   \mu(I_{n})
   &=\mu(J_{k-1})\sum_{w_{n+1}\cdots w_{n_{k}}\in \Xi}e^{-s_{k} m_{k}    -  s_{k}S_{m_{ k}}f(y^{k}) } ,
\end{split}
\end{align*}
where $\Xi$ denote the set of $(w_{n+1}\cdots w_{n_{k}})$ such that $ (w_{1}\cdots w_{n}w_{n+1}\cdots w_{n_{k}} )\in \mathcal{F}_{m_{k}}.$

Firstly, we estimate the last summation,
$$\sum_{w_{n+1}\cdots w_{n_{k}}\in \Xi}e^{-s_{k} m_{k}    -  s_{k}S_{m_{ k}}f(y^{k}) }\leq \sum_{w_{n+1}\cdots w_{n_{k}}\in \Xi}e^{-s_{k} m_{k}    -  s_{k}S_{m_{ k}-l}f(\sigma^{l} y^{k}) }.$$
By  virtue of the definition of $s_{k}$, we have

$$\sum_{w_{1}\cdots w_{l}w_{l+1}\cdots w_{n_{k}}\in \mathcal{F}_{m_{k}}}e^{-s_{k} m_{k}    -  s_{k}S_{m_{ k}}f(y^{k}) }=1.$$
 (1) If $l<M$,
 \begin{align*}
 \begin{split}
 1&\geq \sum_{  w_{1}\cdots w_{l}w_{l+1}\cdots w_{n_{k}}\in \Xi}e^{-s_{k} m_{k}    -  s_{k}S_{m_{ k}}f(y^{k}) }\\
 & = \sum_{ w_{1}\cdots w_{l}w_{l+1}\cdots w_{n_{k}}\in \Xi}e^{-s_{k}l    -  s_{k}S_{l}f(y^{k}) }e^{-s_{k} (m_{k}-l)    -  s_{k}S_{m_{ k}-l}f(\sigma^{l} y^{k}) }\\
 &\geq e^{-s_{k}M-s_{k}M||f||}\sum_{   w_{l+1}\cdots w_{n_{k}}\in \Xi} e^{-s_{k} (m_{k}-l)    -  s_{k}S_{m_{ k}-l}f(\sigma^{l}y^{k}) }.
 \end{split}
 \end{align*}
 Hence,
 \begin{align*}
 \begin{split}
 \sum_{   w_{l+1}\cdots w_{n_{k}}\in \Xi} e^{-s_{k} (m_{k}-l)    -  s_{k}S_{m_{ k}-l}f(y^{k}) }&\leq e^{ s_{k}M+s_{k}M||f||}
 .\end{split}
\end{align*}
As a consequence,

\begin{align*}
\begin{split}
\mu(I_{n})&\leq C(\eta)e^{-(n_{k-1}+t_{k-1})(s(X)-\Delta(\eta))+s_{k}M+s_{k}M||f||}\\&\leq C(\eta){\rm diam}(I_{n})^{s(X)-\Delta(\eta)}e^{M(s(X)-\Delta(\eta))}e^{s_{k}M+s_{k}M||f||}\\
&\leq C(\eta){\rm diam}(I_{n})^{s(X)-\Delta(\eta)}e^{M(s(X)-\Delta(\eta))+(s(X)+\eta)(M+M||f||)}.
\end{split}
\end{align*}
 (2) If $l\geq M$, firstly, assume that $s_{l}(\mathcal{F})$ is well defined, and then we have $|s_{l}(\mathcal{F})-s_{k}|<\eta.$

\begin{align*}
\begin{split}
1&=\sum_{w_{1}\cdots w_{l}w_{l+1}\cdots w_{n_{k}}\in \mathcal{F}_{n_{k}}}e^{-s_{k} m_{k}    -  s_{k}S_{m_{ k}}f(y^{k}) }\\
&=\sum_{ w_{l+1}\cdots w_{n_{k}}}\sum_{w_{1}\cdots w_{l} }e^{-s_{k}l    -  s_{k}S_{l}f(y^{k}) }e^{-s_{k} (m_{k}-l)    -  s_{k}S_{m_{ k}-l}f(\sigma^{l}y^{k}) }.
\end{split}
\end{align*}
Moreover, since $s_{l}(\mathcal{F})$ is well defined,
\begin{align*}
\begin{split}
\sum_{w_{1}\cdots w_{l}\in \mathcal{F}_{l} }e^{-s_{k}l    -  s_{k}S_{l}f(y^{k}) }&\geq \sum_{w_{1}\cdots w_{l}\in \mathcal{F}_{l} }e^{-(s_{l}(\mathcal{F})+\eta)l    -  (s_{l}(\mathcal{F})+\eta)S_{l}f( y^{k}) }\\
&\geq e^{-\eta l-l\eta||f||}.
\end{split}
\end{align*}
On the other hand, if $s_{l}(\mathcal{F})$ can not be defined, we can choose $M+l\geq l^{'}>l$, such that $s_{l^{'}}(\mathcal{F})$  and $|s_{l^{'}}(\mathcal{F})-s_{k}|<\eta$ is well defined.
Moreover, we have

\begin{align*}
\begin{split}
\sum_{w_{1}\cdots w_{l} }e^{-s_{k}l    -  s_{k}S_{l}f(y^{k}) }&\geq \sum_{w_{1}\cdots w_{l^{'}} }e^{-s_{k}l^{'}    -  s_{k}S_{l^{'}}f(y^{k}) } \\
&\geq e^{-\eta l^{
'}-l^{'}\eta||f||}\\
&\geq e^{-\eta l -l \eta||f||-M\eta-M\eta||f||}.
\end{split}
\end{align*}
As a consequence,
\begin{align*}
\begin{split}
 \sum_{ w_{l+1}\cdots w_{n_{k}}\in \Xi}  e^{-s_{k} (m_{k}-l)-s_{k}S_{m_{ k}-l}f(y^{k})  }\leq C^{0}(\eta)e^{ \eta l+l\eta||f||},
\end{split}
\end{align*}
where $C^{0}(\eta):=e^{M\eta+M\eta||f||}$.
It turns out that
\begin{align*}
\begin{split}
 \sum_{ w_{l+1}\cdots w_{n_{k}}\in \Xi}   e^{-s_{k}  m_{k} -s_{k}S_{m_{ k}-l}f(y^{k}) }\leq C^{0}(\eta)e^{ \eta l+l\eta||f||}e^{-s_{k}l} .
\end{split}
\end{align*}
Finally,

\begin{align*}
\begin{split}
\mu(I_{n})&\leq  C^{0}(\eta)C(\eta) e^{-(n_{k-1}+t_{k-1})(s(X)-\Delta(\eta))+\eta l+l\eta||f||-s_{k}l}\\&=  C^{0}(\eta)C(\eta){\rm diam}(I_{n})^{s(X)-\Delta(\eta)}  e^{ l(s(X)-\Delta(\eta))+\eta l+l\eta||f||-s_{k}l}\\
&\leq  C^{0}(\eta)C(\eta){\rm diam}(I_{n})^{s(X) -\Delta(\eta)}  e^{ l(-\Delta(\eta)+2\eta+ \eta||f|| )  } \\
&\leq \min\{C^{0}(\eta)C(\eta){\rm diam}(I_{n})^{s(X)  -2\eta-\eta||f||}, C^{0}(\eta)C(\eta){\rm diam}(I_{n})^{s(X) -\Delta(\eta)}\} .
\end{split}
\end{align*}
By Theorem \ref{thg} and $\eta$ can be arbitrary small, we finish the proof of Theorem \ref{main2}.

\begin{rem}
 In the end we pose a question about Theorem {\rm \ref{main1}}.  Does this result remain valid for $\psi(n)$ has no monotonicity and $\liminf\limits_{n\to\infty}\psi(n)=0$ ?
\end{rem}

 \section{Applications}
 {\bf S-gap shifts}
$S$-gap shift $\sum_{S}$ is a subshift of $\{0,1\}^{\mathbb Z}$ defined by the rule that for a fixed $S\subset \{0,1,2\cdots\}$, the number of $0$ between consecutive $s$ is an integer in $S$. That is, the language
  $$\{0^{n}10^{n_1}0^{n_2}0^{n_3}\cdots0^{n_k}1: 1\leq i\leq k ~\text{and}, n,m\in \mathbb N\},$$
  together with $\{0^{n}:n\in \mathbb N\}$, where we assume that $S$ is infinite.

\noindent{\bf $\beta$-shifts}
  Fix $\beta>1$, write $b=\lceil \beta\rceil, $ and let $w^{\beta}\in \{0,1,\cdots,b-1\}^{\mathbb N}$ be the greedy $\beta$-expansion of $1$. Then $w^{\beta}$ satisfies $\sum_{j=1}^{\infty}w_{j}^{\beta} \beta^{-j}=1,$ and
  has the property that $\sigma^{j}(w^{\beta})\prec w^{\beta}$ for all $j\geq 1$, where $\prec$ denotes the lexicographic ordering. The $\beta$-shift is defined by
  $$\Sigma_{\beta}=\{x\in \{0,1,\cdots,b-1\}^{\mathbb N}:\sigma^{j}(x)\prec w^{\beta} ~\text{for all } j\geq 1\}.$$
In \cite{CliTho}, S-gap shifts, $\beta$ shifts and their factors satisfy the non-unform structure i.e., for $X:=\sum_{S} ~\text{or}~\sum_{\beta}$ there exists $\mathcal{G}\subset \mathcal{L}(X)$ has $(W)$-specification and  $\mathcal{L}(X)$ is edit approachable by $\mathcal{G}$.
Set
 $$ R(\psi):=\{x\in X: d(\sigma^{n}x,x)<\psi(n)~\text{for infinitely many}~n\in \mathbb N\}.$$
 Then we have

$${\rm dim}_{H}R(\psi)=\dfrac{h}{1+b},~\text{with}~b=\liminf\limits_{n\to\infty}\dfrac{-\log \psi(n)}{n}.$$
 Let $f$ be a positive continuous function defined on $X$, set
$$R(f)=\{x\in X:~d(\sigma^{n}x,x)\leq e^{-S_{n}f(x)} ~\text{for infinitely many } ~n\in \mathbb N\}.$$
 The Hasudorff dimension of the set $R(f)$ is the unique solution $s$ of the following pressure equation
$$P(-s(f+1) )=0.$$

\end{document}